\documentclass{elsarticle}
\usepackage{amsmath,amssymb,yhmath}
\usepackage{graphicx}
\usepackage{xcolor,epic, tikz}
\newtheorem{thm}{Theorem}
\newtheorem{lem}[thm]{Lemma}
\newdefinition{defin}{Definition}
\newdefinition{rmk}{Remark}
\newproof{pf}{Proof}
\newproof{pot}{Proof}
\def\C{\hskip 3pt\hbox{$\vrule height 6pt\kern-3ptC$}}

\def\R{\mathbb{R}}

\font\tenneg=cmbx10 \font\sevenneg=cmbx7 \font\fiveneg=cmbx5
\newfam\negfam
\textfont\negfam=\tenneg \scriptfont\negfam=\sevenneg
\scriptscriptfont\negfam=\fiveneg

\font\tenbold=cmmib10 \font\sevenbold=cmmib7 \font\fivebold=cmmib5
\newfam\boldfam
\textfont\boldfam=\tenbold \scriptfont\boldfam=\sevenbold
\scriptscriptfont\boldfam=\fivebold

\font\tenbsy=cmbsy10 \font\sevenbsy=cmbsy7 \font\fivebsy=cmbsy5
\newfam\bsyfam
\textfont\bsyfam=\tenbsy \scriptfont\bsyfam=\sevenbsy
\scriptscriptfont\bsyfam=\fivebsy

\begin{document}

\begin{frontmatter}

\title{On the order of Runge Kutta methods  reusing last stage \tnoteref{t1}}

\tnotetext[t1]{This work was supported by Ministerio de Ciencia, Innovaci\'on y Universidades, Agencia estatal de Investigaci\'on, Spain, under the project PID2022–141385NB–I00.}

 \author[label1]{M. Calvo}
 \ead{calvo@unizar.es}

 \author[label1]{J.I. Montijano\corref{cor1}}
 \ead{monti@unizar.es}

 \author[label1]{L. R\'{a}ndez}
 \ead{randez@unizar.es}

 \cortext[cor1]{Corresponding author}

 \address[label1]{Departamento  Matem\'{a}tica Aplicada, Universidad de Zaragoza, 50009-Zaragoza, Spain.}

%\baselineskip=0.9\normalbaselineskip \vspace{-3pt}

\begin{keyword}


\MSC 65L05 \sep 65L07
\end{keyword}

\begin{abstract}

In this paper we consider explicit Runge--Kutta (RK) methods 
for the numerical solution of Initial Value Problems (IVPs) in differential equations
in which the last function evaluation  in a step
is reused, substituting the first evaluation for the next step.
This requirement implies that, except for the first step, the 
computational cost of an $s$--stage explicit RK method reduces from $s$ to
$(s-1)$ function evaluations per step.
It will be seen that, in general,  if a RK method has order $p$, when
introducing this reused stage the new method has in general order $<p$.

In this context we study the conditions on the coefficients
of an explicit RK method with order $p$ such that the global error of the method 
reusing the last stage has the same order $p$ of accuracy.
A detailed study on the conditions that the coefficients of an $s$ stage
method with order $p$ must satisfy for the method reusing the last stage has also order $p$.
In particular, for the families of $s$--stage explicit RK methods with the highest order $p=s$ we identify 
those methods reusing the last stage which retain 
the order of of the original method.
The results of some numerical experiments are presented to confirm 
the theoretical results.
\end{abstract}

\end{frontmatter}

\section{Introduction}

Runge Kutta (RK) methods for the numerical solution of IVPs in systems 
of differential equations
\begin{equation}\label{A1}
\dfrac{d y(t)}{dt}= f(t, y(t)),
\quad
y(0)= y_0 \in \R^d,
\quad 
t \in [0,T]
\end{equation}
are among the most popular algorithms for the numerical solution of (\ref{A1}). In particular  explicit $s$--stage RK methods advance the numerical solution of \eqref{A1} from $ (t_n=n h, y_n)$ to 
$ (t_{n+1}= t_n+h, y_{n+1})$ with the formulas
\begin{equation}\label{A2}
y_{n+1} = y_n + h \; \left[
b_1 \; f_{n,1} + b_2 \; f_{n,2}+ \ldots + b_n \; f_{n,s}
\right],
\end{equation}
where
\begin{eqnarray}
f_{n,1} &=& f( t_n , y_n)
\nonumber
\\
f_{n,2} &=& f( t_n + c_2 h , y_n + h \; a_{21} f_{n,1})
\nonumber
\\
\vdots & \vdots
\label{A3}
\\
f_{n,s} &=& f( t_n + c_s h, y_n + h a_{s,1} f_{n,1}+ \ldots
+h a_{s,s-1} f_{n, s-1} ).
\nonumber
\end{eqnarray}
It is common to collect the available parameters of (\ref{A2}),(\ref{A3}) , $ b_j, c_j, a_{jk} $ in the form of the Butcher array
\cite{Butcher:16, HNW}

\begin{equation}\label{A4}
\begin{array}{cccccc}
c_1=0 & \vline & 0 &  &  &  \\ 
c_2 & \vline & a_{21} & 0 &  &  \\ 
\vdots & \vline & \vdots &  &  &  \\ 
\vdots  & \vline & \vdots &  &  &  \\ 
c_s & \vline  & a_{s1} & \ldots & a_{s,s-1} & 0 \\ 
\hline
 & \vline & b_1 & b_2  & \ldots & b_s
\end{array} 
\qquad
\begin{array}{ccc}
 c & \vline & A
 \\
 \hline
 0 & \vline & b^T
\end{array}
\end{equation}
with the condition 
\[
 c_i = \sum_{j=1}^{i-1} \;  a_{ij}
 \qquad
  i=1,..,s
\]

Usually the free parameters of (\ref{A4}) are selected in such a way that the method achieves the highest accuracy possible for all
IVP (\ref{A1}) with $f(t,y)$ sufficiently smooth.
Denoting by $y(t; t_n,y_n)$ the local solution of differential problem $y'(t) = f(t, y(t))$ 
that passes by the point $(t_n, y_n)$, that is,  $y(t_n)=y_n$,
the accuracy is measured by the local error, defined  as the error after
one step $\Vert y_{n+1} - y(t_n+h; t_n,y_n) \Vert$.
Then the method has order $p$ \cite{HNW} if
\begin{equation}\label{A5}
\Vert y_{n+1} - y(t_n+h; t_n,y_n) \Vert=
\mathcal{O} (h^{p+1}),
\quad h \to 0
\end{equation}
for all IVP (\ref{A1}).
This condition implies that the global error of (\ref{A1}) behaves as
$ \Vert y(T) - y_N \Vert = \mathcal{O} (h^{p})$ with $ N h = T$.
In this setting it is well known that for $ s \le 4$ it is possible to
get order $ p=s$ whereas for $ s>4$ the maximum order $ p < s$.

Clearly the computational cost of an explicit RK method is associated to the number of function evaluations at each step i.e. the number of stages. With the aim to reduce this cost we examine here explicit RK methods with $c_s=1$ in which the first evaluation of the $n$-th step,
$f_{n,1}=f(t_n,y_n)$, is substituted by the last evaluation of the 
$(n-1)$-th step, $f_{n-1,s}$ i.e.
$f_{n,1}=f( t_{n-1}+ c_s \; h, y_{n-1}+h\sum_{j=1}^{s-1} a_{s,j} f_{n-1,j})$.
These methods will be called explicit RK methods reusing last stage, and we will denote them by RKRLS.
Thus instead of (\ref{A2}),(\ref{A3}) we have the method

\begin{equation}\label{A6}
y_{n+1} = y_n + h \; \left[
b_1 \; f_{n,1} + b_2 \; f_{n,2}+ \ldots + b_n \; f_{n,s}
\right],
\end{equation}
with
\begin{eqnarray}
f_{n,1} &=& f_{n-1,s}  = f( t_{n-1}+ c_s \; h, y_{n-1} + h \sum_{j=1}^{s-1} a_{s,j} f_{n-1,j})
\nonumber
\\
f_{n,2} &=& f( t_n + c_2 h , y_n + h \; a_{21} f_{n,1})
\nonumber
\\
\vdots & \vdots
\label{A7}
\\
f_{n,s} &=& f( t_n + c_s h, y_n + h a_{s,1} f_{n,1}+ \ldots
+h a_{s,s-1} f_{n, s-1} ).
\nonumber
\end{eqnarray}

In this setting (except for the first step) we only have $(s-1)$
function evaluations per step.
However, even in the case that the explicit RK method (\ref{A2}),(\ref{A3}) has order $p$,
the associated RKRLS method (\ref{A6})(\ref{A7})  will not have in general the same convergence order. 

The method (\ref{A6})(\ref{A7}) advances the numerical solution from $t_n$ to $t_{n+1}=t_n+h$
by considering at $t_n$ two approximations to the solution:
$y_n \simeq y(t_n)$
and
$f_{n,1}\equiv f_{n-1,s}= f(t_{n-1} +h, y_{n-1}+ h \sum_{j=1}^{s-1} a_{s,j} f_{n-1,j})
\simeq y'(t_n)$, obtaining with them the next two approximations $y_{n+1}$ and $f_{n+1,1}$.
The local error cannot be defined as in standard one step methods
based only in the approximation $y_n$, but a two-approximations, one step formulation can be used
to apply the theory  in \cite{HNW} (Section III.8, p. 436).
Here for technical reasons, instead of  $f_{n-1,s}$ we will use the quantity
$y_{n,1}$ defined by
\[
y_{n,1}  \equiv y_{n-1} + h \sum_{j=1}^{s-1} a_{s,j} f_{n-1,j}.
\]
Then the method (\ref{A6})(\ref{A7}) can be written as a two-approximations, one step method in the form
\begin{equation}
\label{asonestep}
\begin{pmatrix}  y_{n+1} \\  y_{n+1,1} \end{pmatrix} =
\begin{pmatrix}  y_{n} + h \sum_{i=1}^s  b_i f_{n,i} \\[6pt] 
y_{n} + h \sum_{i=1}^{s-1}  a_{si} f_{n,i} \end{pmatrix} =
\begin{pmatrix} 1 & 0  \\  1 & 0 \end{pmatrix} 
\begin{pmatrix}  y_{n} \\  y_{n,1}\end{pmatrix} +
h \Phi(t_n, y_n, y_{n,1},h) 
\end{equation}
where  $f_{n,1}\equiv f(t_n, y_{n,1})= f_{n-1,s}$ depends on the second carried vector $y_{n,1}$ and  is in fact the last stage of the previous step. Since this value has been computed in the previous step, it does not need to be computed in this step.

The vector $y_{n+1,1}$ is an approximation to the  exact value  $y(t_{n} + h)$. Thus, the order of the method can be studied by means of standard Butcher series  as in
\cite{HNW} (Section III.8, p. 436). Unfortunately, $y_{n+1,1}$ usually approximates $y(t_{n} + h)$  with order $q<p$ and this theory will ensure that the method has order $\min(q+1,p)$.  To get order $p$, $y_{n+1,1}$ should have order $p-1$ at least and this can be
a very restrictive condition. 

Two-step Runge-Kutta methods were considered for the first time in \cite{ByLa:66}, where they were named pseudo-Runge-Kutta methods. They have been considered later, for example in \cite{HW:73} and \cite{JackZen:95}. In both cases the 
new approximation $y_{n+1}$ can be expressed as
\begin{equation}
\label{manystepsmethod}
y_{n+1}= y_n + h \Phi(t_n,y_n, v_{1n}, v_{2n}, \ldots)
\end{equation}
where $v_{jn}$ are approximations to the solution $y(t)$  of \eqref{A1} at certain previous points $t_{n-1} + a_j h$, of order $p$.
In \cite{Gruttke:70} and \cite{CaCoCo:90} the computation of $y_{n+1}$ uses the stages of the previous step so that it can be expressed as
\eqref{manystepsmethod} but now $v_j$ are approximations to $y'(t_{n-1} + a_j h)$ of order $p-1$.
In all these cases the theory in \cite{HNW} gives order $p$ for the resulting methods.

In \cite{JackTra:95} general two-steps Runge-Kutta methods were studied. Since the approach in \cite{HNW} could not be directly applied, the authors turned to an approach due to Albrecht \cite{Albrecht:85, Albrecht:87} to derive order conditions for the method.  Nevertheless, in \cite{HaWa:97} it was shown that  the results in \cite{HNW} can also be applied to general two-step Runge-Kutta methods formulated as \eqref{manystepsmethod} by assuming that  $v_{in}$ approximates certain Butcher series 
$B(\Phi_i, h)$ evaluated at the solution $y(t)$ at $t_n$.  Here, we will  study the accuracy order of RKRLS methods using this last approach. Starting with a standard RK method order $p$ we derive additional conditions on their coefficients that ensure a global error of the RKRLS of order $p$ even though the approximation $y_{n+1,1}$ has order $q < p-1$.

The paper is organized as follows: In Section 2 we present the theoretical results on the order of convergence of Runge-Kutta methods reusing the last stage. We define the local order of the methods and we prove that local order $p$ implies global convergence of order $p$.
Then we obtain conditions on the coefficients of the RK to have RKRLS method of order $p$, for $p=2,3,4,5$.
In Section 3 we obtain RKRLS methods with orders from 2 to 5, showing that there exist methods with order $p$ and $s=p$ stages ($p-1$ effective stages) for $p \le 4$ and order $5$ with $s=6$ stages.
In Section 4 we verify the theoretical results by means of some numerical experiments.
Finally, in Section 5 we present some conclusions and future work.

\section{Order and convergence}
\subsection{Order theory}

For a Runge-Kutta method $(A,b)$ with $s$ stages, $n$ consecutive steps of size $h$ are equivalent to one step of size $H=nh$ with the $n\times s$ stages composed method $(A^{(n)}, b^{(n)})$ with coefficients
\[
A^{(n)}=\dfrac{1}{n}
\begin{pmatrix}
A  &   \\
eb^T  &  A  \\
\cdots & \cdots & \ddots  \\
e b^T & \cdots  & e b^T &  A
\end{pmatrix},
\qquad   b^{(n)}= \dfrac{1}{n}
\begin{pmatrix}
b   \\
b \\
\vdots   \\
b
\end{pmatrix},
\]
where $e= (1, \ldots, 1)^T \in \R^s$.
Note that if the RK base method $(A, b)$ has order 1, since $Ae=c$, the composed method satisfies
\[
c^{(n)} = A^{(n)}  e^{(n)}=\dfrac{1}{n}
\begin{pmatrix}
c  \\
e+c  \\
\vdots   \\
(n-1)e +c
\end{pmatrix}.
\]
where $e^{(n)}= (1, \ldots, 1)^T \in \R^{ns}$.
Here, a superindex $(n)$ represents vectors and matrices of dimension $sn$.

It can be easily proved that the RK method $(A,b)$ has local order $p$ if and only if the composed method 
$(A^{(n)}, b^{(n)})$ has also the same order for any $n\ge2$.
Then, the local error of $(A,b)$ behaves as $\mathcal{O}(h^{p+1})$, and
the local error of the n-step composed method $(A^{(n)}, b^{(n)})$ behaves as $\mathcal{O}(H^{p+1})$.

For a RKRLS method, the first stage is substituted by the last one of the previous step, that is, at the step from $t_n$ to $t_{n+1}=t_n+h$, $f_{n,1}= f_{n-1,s} = h f(y_{n-1} + ha_{s,1}f_{n-1,1} + \ldots + h a_{s,s-1} f_{n-1,s-1})$. The first two steps of the method are equivalent to one step of the $2s$ stage method with coefficients
\[
\widehat A^{(2)}=\dfrac{1}{2}
\begin{pmatrix}
0  &    & & & & | &\\
a_{21} &  0  & & & & |\\
\cdots & \cdots & \ddots &\ddots  & & |\\
a_{s1} & a_{s2}& \cdots & a_{s,s-1} & 0 & | & 0 \\[3pt]
\hline
a_{s1} & a_{s2}& \cdots & a_{s,s-1} & 0 &  | & 0  & 0\\
b_1 & b_2 & \cdots & b_{s-1} & b_s  &  | & a_{21} &0 \\
\cdots & \cdots & \ddots &\ddots  & & | \\
b_1 & b_2 & \cdots & b_{s-1} & b_s  & | & a_{s1} & \cdots & a_{s,s-1} & 0 \\
\end{pmatrix},
\  \widehat b^{(2)}= \dfrac{1}{2}
\begin{pmatrix}
b   \\
b \\
\end{pmatrix}.
\]
that we will name 2-step composed method.
Observe that the rows $s$ and $s+1$ of the matrix $\widehat A^{(2)}$ are identical, which means that the method is reducible and the number of function evaluations required is in practice $2s-1$, in agreement with the saving of one evaluation at each step.

In general, $n$ consecutive steps of size $h$ of the RKRLS method are equivalent to one step of size $H=nh$
with the corresponding  $n$-step composed method  $(\widehat A^{(n)}, \widehat b^{(n)})$ given by
\begin{equation}
\label{nreused}
\widehat A^{(n)}=\dfrac{1}{n}
\begin{pmatrix}
A  &   \\
E &  A  \\
e b^T &   E & A \\
\cdots & \cdots & \ddots &\ddots  \\
e b^T &  \cdots & e b^T & E & A
\end{pmatrix},
\  \widehat b^{(n)}= \dfrac{1}{n}
\begin{pmatrix}
b   \\
b \\
\vdots   \\
b
\end{pmatrix},
\end{equation}
where $E=e b^T+M$, $M=e_1 (a-b)^T$, $a=( a_{s,1}, \ldots, a_{s,s-1}, 0)^T$ is the last row of the matrix $A$ and
$e_1= (1, 0, \ldots, 0)^T\in \R^s$.
Clearly, since $a^T e=1$, $E e=e$ and $\widehat c^{(n)} = \widehat A^{(n)}  e^{(n)}= c^{(n)}$.

The standard definition of local error of order $p$ can not be directly applied to RKRLS methods
because it does not allow us to obtain local order $p$. Then, using the composed methods
we define the local error of these methods in the following way
\begin{defin}
\label{def1}
We will say that a RKRLS method has local order $p$ if the $n$-composed method $(A^{(n)}, b^{(n)})$ has local order $p$ (as defined in \cite[Definition 1, p. 134]{HNW}) for any $n\ge 2$.
\end{defin}

Let us see now that, with the above definition, if a RKRLS method has local order $p$, then it is convergent of order $p$, that is, the global error satisfies $y_n - y(t_n)= \mathcal{O}(h^p)$.  

Hereafter we will assume, without loss of generality, that the differential problem is autonomous, that is $f=f(y)$.

Assume that the method has order $p$ and let us consider the composed method with $r$ steps (we will use the case $r=p-1$ later). Starting from the initial point $(t_0, y_0)=(t_0, y(t_0))$ we get $y_{r}$, approximation of order $p$ to $y(t_0+rh)$ and $y_{r,1}=y_{r-1} + h a_{s1}f_{r-1,s} + h a_{s2} f_{r,2} + \ldots + h a_{s,s-1} f_{r,s-1}$ that approximates $y(t_0+rh)$ in general with order $1$. This means that 
$y_r- y(t_0+rh) = \mathcal{O}(h^{p+1})$ and $y_{r,1} - y(t_0+rh) = \mathcal{O}(h^2)$. Using that the approximation 
$y_{r,1}$ is in fact the approximation provided by the method $(A^{(r)}, (b^T,\ldots,b^T, a^T)^T)$ we can obtain easily the coefficients of the Butcher series of $y_{2,1}$ around the initial point $(t_0, y_0)$ given by
\[
y_{r,1} = y_0 + \dfrac{1}{r}(b^T e + a^T e) f(y_0) (rh)+ 
\dfrac{1}{2!} 2\dfrac{1}{r^2} (b^Tc +a^T (e+c)) f'(y_0) f(y_0) (rh)^2 + \cdots
\]
To study the effect of this approximation in the next step, it is more convenient to have the expansion of $y_{r,1}$
around the current point $(t_r,y_r)$. This can be obtained by means of the adjoint of the method $(A^{(r)}, b^{(r)})$  (see \cite{HNW}, p. 219). This method is given by
\begin{equation}
\label{adjointmethod}
\begin{array}{l}
\bar A^{(r)} = e^{(r)} {b^{(r)}}^T - A^{(r)}= \dfrac{1}{r}
\begin{pmatrix}
e b^T -A  &  e b^T & & \cdots & e b^T \\[5pt]
 - M & e b^T -A   & & \cdots& e b^T\\[5pt]
0 &  - M  & \ddots & & e b^T\\
\vdots  &  \vdots &  \ddots & \ddots \\
0 & 0 &  \cdots  & - M & e b^T -A  
\end{pmatrix} \\[45pt]
\bar c^{(r)} =   \bar A^{(r)} e^{(r)} = \dfrac{1}{r}
\begin{pmatrix}
re-c  \\
\vdots \\
e-c\\
\end{pmatrix},
\end{array}
\end{equation}
and the approximation $y_{r,1}$ (it is in fact the last stage of the adjoint method) is given by the method 
$(\bar A^{(r)}, \bar b^{(r)})$  where
\begin{equation}
\label{badjointmethod}
\bar b^{(r)} \equiv  ((e_s^{(r)})^T \bar A^{(r)})^T = \dfrac{1}{r}
\begin{pmatrix}
0  \\
\vdots \\
b -a\\
\end{pmatrix},
\end{equation}
applied with step $-H=-rh$. Computing the Butcher series for this method we get
\begin{equation}
\label{seriesyr1}
\begin{array}{rl}
y_{r,1}= & y_{r} + \sum_{\tau} \dfrac{\alpha(\tau)}{\rho(\tau)!} \gamma(\tau)  \phi_r(\tau) F(\tau)(y_{r}) (-r h)^{\rho(\tau)} \\[8pt]
= & y_{r} + \sum_{\tau} \dfrac{\alpha(\tau)}{\rho(\tau)!} \gamma(\tau)  \varphi_r(\tau) F(\tau)(y_{r}) h^{\rho(\tau)}
\end{array}
\end{equation}
where (see e.g. \cite[Section II.2, p. 145]{HNW}, p. 145 or \cite[Chapter 3, p 143]{Butcher:16}) $F(\tau)(y_{r})$ is the elementary differential associated to the tree $\tau$ evaluated at the point $y_{r}$, $\alpha(\tau)$ is a symmetry factor, $\gamma(\tau)$
is a density factor, $\rho(\tau)$ is the order of the tree and 
$\phi_r(\tau)= (-r)^{\rho(\tau)}\varphi_r(\tau)$ is the coefficient, method dependent,  associated to $\tau$. The value of $\varphi_r(\tau)$ 
for the first trees $\tau$ are given in Table \ref{tableexpansion}. Note the minus signs in the coefficients of odd powers of $h$ introduced by the negative step $-H$.

\begin{table}
\label{tableexpansion}
\begin{centering}
\begin{tabular}{l| @{\hspace{20pt}}l}
\hskip 0.6cm$F(\tau)$  &  \hskip 1.5cm $\varphi_r(\tau)$  \\[2pt]
\hline
$f$                   &  0  \\
$f'f$                 &  $(b-a)^T (e-c)$ \\[2pt]
$f''(f,f)$            &  \hspace{-8pt}$-(b-a)^T (e-c)^2$ \\[2pt]
$f'f'f$               &  \hspace{-8pt}$-(b-a)^T[ - M (2e-c) +(eb^T -A) (e-c)]$\\[2pt]
$f'''(f,f,f)$         &  $(b-a)^T (e-c)^3$ \\[2pt]
$f''(f'f,f)$          &  $(b-a)^T([- M (2e-c) + (eb^T -A) (e-c)]\cdot (e-c))$\\[2pt]
$f'f''(f,f)$          &  $(b-a)^T [- M (2e-c)^2+ (eb^T -A) (e-c)^2]$ \\[2pt]
$f'f'f'f$             &  $(b-a)^T [ (- M + eb^T -A)^2 (e-c)+ M(eb^T -A)e]$ \\[2pt]
$f^{\text{iv}}(f,f,f,f)$ &  \hspace{-8pt}$-(b-a)^T (e-c)^4$ \\[2pt]
\ldots 
\end{tabular}
\end{centering}
\caption{Coefficients of the Butcher expansion of the last stage of the RK}
\end{table}

The key point in this study is that  the first coefficients of the Butcher series of $y_{r,1}$ and $y_{r+1,1}$ are equal.  

\begin{lem}
\label{lema1}
The coefficients of the Butcher series \eqref{seriesyr1} satisfy $\varphi_{r-1}(\tau)=\varphi_{r}(\tau)$ for all tree $\tau$ with order $\rho(\tau) \le r$, 
\end{lem}
\begin{pf}
The coefficient $\varphi_r(\tau)$ is given by
\[
\varphi_r(\tau)= (-r)^{\rho(\tau)} \bar b^{(r)} \Phi_r(\tau)
\]
where $\Phi_r(\tau)$ is a vector depending on the matrix $\bar A^{(r)}$ and vector $\bar c^{(r)}$ given in \eqref{adjointmethod}, and is defined recursively by
\begin{itemize}
\item
For the unique tree $\tau_0$ of order one, $\Phi_r(\tau_0)= e^{(r)}$ 
\item
For the unique tree $\{\tau_0\}$ of order two, $\Phi_r(\{\tau_0\}) = A^{(r)} \Phi_r(\tau_0)= \bar c^{(r)}$
\item
If   $\tau= \{\tau_1, \ldots, \tau_k\}$, then   
$\Phi_r(\tau) = A^{(r)} \Phi_r(\tau_1)\cdot \ldots \cdot A^{(r)} \Phi_r(\tau_k)$
\end{itemize}
Note that the vector $\bar b^{(r)}$ carries a factor $1/r$ and  the vector $\Phi_r(\tau)$ carries a factor 
$1/r^{\rho(\tau)-1}$. Consequently, the coefficient $\varphi_r(\tau)$ does not depend explicitely on $r$.

Let us see that for  $\rho(\tau)=j \le r$ the last $(r-j)s$ components of the vectors 
$\Phi_{r-1}(\tau)$ and $\Phi_r(\tau)$ are equal, that is, the last $(r-j)$, $s$-dimensional, blocks of the vectors are equal.

First, it is easy to verify that if the last $m$ blocks of the vectors 
$r^{\rho(\tau))-1} \Phi_{r-1}(\tau)$ and $r^{\rho(\tau))}\Phi_r(\tau)$ are equal, then the last $m-1$ blocks 
of the vectors 
$r^{\rho(\tau))}\bar A^{(r-1)}\Phi_{r-1}(\tau)$ and $r^{\rho(\tau))+1}\bar A^{(r)}\Phi_r(\tau)$ are equal.

For $\rho(\tau)=1$, there is a unique tree, $\tau_0$,  and $\Phi_{r-1}(\tau_0)=e^{(r-1)}$, 
$\Phi_{r}(\tau_0)=e^{(r)}$,  and the last $r-1$ blocks  of these vectors are equal.
For $\rho(\tau)=2$, there is a unique tree, $\{\tau_0\}$,  and $\Phi_{r-1}(\{\tau_0\})=\bar c^{(r-1)}$, 
$\Phi_{r}(\{\tau_0\})=\bar c^{(r)}$,  and also the last $r-1$ blocks  of the vectors $(r-1)\Phi_{r-1}(\{\tau_0\})$ 
and $r\Phi_{r}(\{\tau_0\})$ are equal (see equation \eqref{adjointmethod}).

By induction on $j$, assume that the property is true up to $\rho(\tau)=j$. If $\tau$ has order $j+1$, $\rho(\tau)=j+1$,
and it must be $\tau=\{\tau_1, \ldots, \tau_k\}$, where $\rho(\tau_i) \le j$ for all $i= 1,\ldots, k$. Then, 
$(r-1)^{\rho(\tau_i)-1}\Phi_{r-1}(\tau_i)$ and $r^{\rho(\tau_i)-1}\Phi_r(\tau_i)$ have at least the last 
$r-j+1$ blocks equal, which implies that
$(r-1)^{\rho(\tau_i)}A^{(r-1)}\Phi_{r-1}(\tau_i)$ and $r^{\rho(\tau_i)}A^{(r)}\Phi_r(\tau_i)$ have the last $r-j$ blocks equal. Consequently
$\Phi_{r-1}(\tau)$ and $\Phi_r(\tau)$ have the last $r-(j+1)$ blocks equal. This proves the property up to $j=r$.

Since the vectors
$\bar b^{(r-1)}$ and $\bar b^{(r)}$ have all the $s$-dimensional blocks zero except the last one, it is clear that 
\[
(r-1)^{\rho(\tau)}\bar b^{(r-1)} \Phi_{r-1}(\tau) = r^{\rho(\tau)}\bar b^{(r)} \Phi_r(\tau), \quad \forall  \rho(\tau) \le r.
\]
\end{pf}

Note that in general $\varphi_r(\tau) \ne \varphi_{r+1}(\tau)$ for  $\rho(\tau) > r$.  This is the reason why we will need to  consider the composed methods with $r=p-1$ and $r+1=p$ to guarantee that both series have the same coefficients up to order $p$, in order to prove the convergence of order $p$.

Let us define the function $z(t)$ as the Butcher series \eqref{seriesyr1} of $y_{p-1,1}$ associated to the adjoint method  \eqref{adjointmethod} with $r=p-1$, evaluated at the point $y(t)$
\[
z(t)= y(t) + \sum_{\tau} \dfrac{\alpha(\tau) }{\rho(\tau)!} \gamma(\tau) \varphi_{p-1}(\tau)  F(\tau)(y(t)) h^{\rho(\tau)}.
\]
We can state the following
\begin{thm}
If a RKRLS method has local order $p$, then it is convergent of order $p$.
\end{thm}
\begin{pf}
We can now follow the approach in \cite{HNW} (Section III.8, p. 436) for the one-step method \eqref{asonestep}. Let us take
$y_n=y(t_n)$ and $y_{n,1}= z(t_n)$. We will see that $y_{n+1}-y(t_{n} +h) = \mathcal{O}(h^{p+1})$ and 
$y_{n+1,1} - z(t_n +h) = \mathcal{O}(h^{p+1})$, that implies that the global error behaves as $\mathcal{O}(h^p)$
and therefore the method is convergent with order $p$.

First, note that by definition, $z(t_n)$ is precisely the last stage of the adjoint method  \eqref{adjointmethod} with $r=p-1$. Then, there is a starting value $y_{n-p}$ such that the $(p-1)$-step composed method gives $y_n=y(t_n)$ and $y_{n-1,s}=y_{n,1}=z(t_n)$. The next approximations
$y_{n+1}$ and $y_{n+1,1}$ are computed with the next step, or, equivalently, with the $p$-step composed method starting at the point 
$(t_n-(p-1) h, y_{n-p+1})$.  Let us denote $y(t; t_n-(p-1)h, y_{n-p+1})$ the solution of the differential equation that at $t_{n-p+1}= t_n -(p-1)h$
has a value $y_{n-p+1}$ (see Figure \ref{figura1})

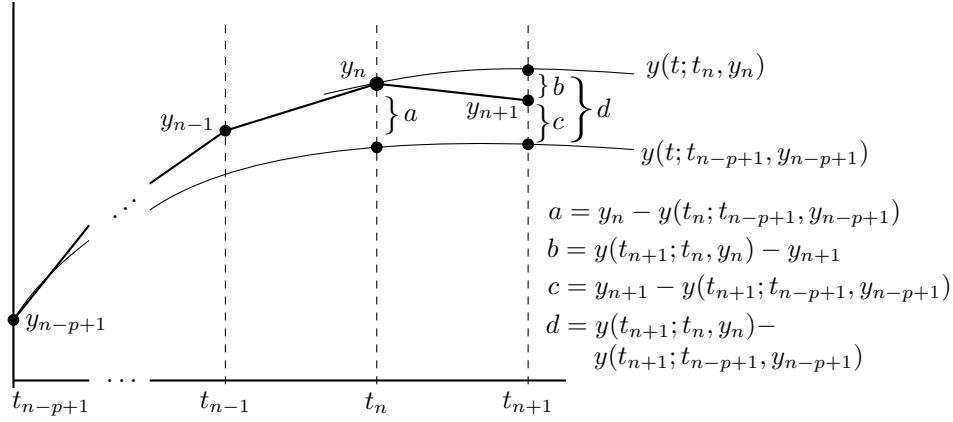
\begin{figure}
\begin{center}
\begin{tikzpicture}   
\draw[thick] (-0.8,0.7) -- (0.2,0.7);  
\draw node at (0.66, 0.7) {\ldots};
\draw[thick] (1.,0.7) -- (6.5,0.7);  
\draw[thick]  (-0.8,0.6) -- (-0.8,5.7);  
\draw node at (-0.3, 0.4) {$t_{n-p+1}$};
\draw[dashed] (2,0.65) -- (2,5.5);  
\draw node at (2, 0.4) {$t_{n-1}$};
\draw[dashed] (4,0.65) -- (4,5.5);  
\draw node at (4, 0.4) {$t_{n}$};
\draw[dashed] (6,0.65) -- (6,5.5);  
\draw node at (6, 0.4) {$t_{n+1}$};
\draw[thick]  (-0.8,1.5) -- (0.2,2.75); 
\draw node at (0.66, 3.07) {$\adots$};
\draw[thick]  (1,3.3) -- (2,4) -- (4,4.62) -- (6,4.4); 
\draw node at (-0.1, 1.45) {$y_{n-p+1}$};
\draw node at (1.5, 4.1) {$y_{n-1}$};
\draw node at (3.7, 4.8) {$y_{n}$};
\draw node at (5.54, 4.25) {$y_{n+1}$};
\draw node at (8.35, 4.85) {$y(t; t_n, y_n)$};
\draw node at (9, 3.70) {$y(t; t_{n-p+1}, y_{n-p+1})$};
\draw node at (4.2, 4.2) {\Large \}};
\draw node at (4.45, 4.2) {$a$};
\draw node at (6.2, 4.59) {\small \}};
\draw node at (6.4, 4.59) {$b$};
\draw node at (6.2, 4.1) {\Large \}};
\draw node at (6.4, 4.1) {$c$};
\draw node at (6.7, 4.28) {\Huge \}};
\draw node at (7.0, 4.30) {$d$};
\draw node at (8.6, 2.92) {$a= y_n-y( t_n; t_{n-p+1}, y_{n-p+1})$};
\draw node at (8.2, 2.42) {$b=y( t_{n+1}; t_{n}, y_{n}) - y_{n+1}$};
\draw node at (8.93, 1.92) {$c= y_{n+1}-y( t_{n+1}; t_{n-p+1}, y_{n-p+1})$};
\draw node at (7.76, 1.42) {$d= y(t_{n+1}; t_n, y_n)-$};
\draw node at (8.65, 1.0) {$y(t_{n+1}; t_{n-p+1}, y_{n-p+1})$};
\filldraw[color=black] (-0.8,1.5) circle (2pt);  
\filldraw[color=black] (2,4) circle (2pt);
\filldraw[color=black] (4,4.62) circle (2.5pt);
\filldraw[color=black] (6,4.4) circle (2pt);
\filldraw[color=black] (4,3.78) circle (2pt);
\filldraw[color=black] (6,3.82) circle (2pt);
\filldraw[color=black] (6,4.8) circle (2pt);
\draw (-0.8,1.5) .. controls (-0.5,2.) and (0.,2.4) .. (0.2,2.55); 
\draw (1,2.95) .. controls (2,3.65) and (4,4) .. (7.4,3.75); 
%\draw (0,2) .. controls (2,3.78) and (4,4) .. (7.4,3.75); 
%\draw (0,2.85) .. controls (2,4.55) and (4,4.85) .. (6.5,4.60); 
\draw  (3.3,4.48) .. controls (4.7,4.8) and (5.4,4.9) .. (7.4,4.75); 
\end{tikzpicture}

\end{center}
\caption{Local error analysis at $(t_n, y_n)$}\label{fig121}
\label{figura1}
\end{figure}

Since the  composed methods $(A^{(p-1)}, b^{(p-1)})$ and  $(A^{(p)}, b^{(p)})$ have order $p$, we can ensure that
\[
\begin{array}{l}
y_n = y(t_n) =  y(t_{n}; t_{n-p+1}, y_{n-p+1}) + \mathcal{O}(h^{p+1})\\[7pt]
y_{n+1} = y(t_{n+1}; t_{n-p+1}, y_{n-p+1}) + \mathcal{O}(h^{p+1}).
\end{array}
\]
Moreover, there exist a Lipschitz constant $L$ such that
\[
\Vert y(t_{n+1}; t_{n-p+1}, y_{n-p+1})- y(t_{n+1}; t_n, y_n) \Vert \le L \Vert  y_n - y(t_{n}; t_{n-p+1}, y_{n-p+1}) \Vert
=\mathcal{O}(h^{p+1}).
\]
Therefore
\[
\begin{array}{l}
y_{n+1}- y(t_{n+1})= y_{n+1} - y(t_{n+1}; t_n, y_n) \\[6pt]
\quad =y_{n+1} - y(t_{n+1}; t_{n-p+1}, y_{n-p+1})+y(t_{n+1}; t_{n-p+1}, y_{n-p+1})-y(t_{n+1}) = \mathcal{O}(h^{p+1})
\end{array}
\]
On the other side,  by definition of $y_{n,1}=z(t_n)$
\[
y_{n,1}= y_{n} + \sum_{\tau} \dfrac{1}{\rho(\tau)!} \gamma(\tau)  
\varphi_{p-1}(\tau) F(\tau)(y_{n+1}) h^{\rho(\tau)} 
\]
and
\[
\begin{array}{l}
y_{n+1,1}= y_{n+1} + \sum_{\rho(\tau)\le p} \dfrac{1}{\rho(\tau)!} \gamma(\tau)  
\varphi_{p}(\tau) F(\tau)(y_{n+1}) h^{\rho(\tau)} + \mathcal{O}(h^{p+1})\\
z(t_{n+1})= y(t_{n+1}) + \sum_{\rho(\tau)\le p} \dfrac{1}{\rho(\tau)!} \gamma(\tau)  
\varphi_{p-1}(\tau) F(\tau)(y(t_{n+1})) h^\rho(\tau) + \mathcal{O}(h^{p+1})
\end{array}
\]
and taking into account that from Lemma \ref{lema1} $\varphi_{p-1}(\tau)=\varphi_{p}(\tau)$ for $\rho(\tau)\le p$, 
and  that $y_{n+1}-y(t_{n+1})=\mathcal{O}(h^{p+1})$,  
$F(\tau)(y_{n+1})- F(\tau)(y(t_{n+1}))= \mathcal{O}(h^{p+1})$, we deduce that
\[
y_{n+1,1} - z(t_{n+1}) = \mathcal{O}(h^{p+1})
\]
\end{pf}

\begin{rmk}\vphantom { 99 }
In the above proof, it was only used that the composed methods with $p-1$ and $p steps$ have local order $p$. Therefore, in the 
Definition \ref{def1} it would be enough to require local order $p$ for $n=p-1$ and $n=p$.  It is possible to prove that if the composed methods have order $p$ for $n=p-1,p$ then 
the RKRLS method has order $p$ for all $n$ (see next subsection). Then both definitions are equivalent.
\end{rmk}

\subsection{Order conditions}
\label{orderconditions}
In order to find the conditions that the coefficients of the modified method must satisfy to have order $p$ we 
will impose that the composed method has order $p$. First let us note that the matrices $A^{(n)}$ 
and $\widehat A^{(n)}$ are related by
\begin{equation}
\label{matrizm}
\widehat A^{(n)} =
A^{(n)} + \dfrac{1}{n}
\begin{pmatrix}
0  &   \\
M  &  0  \\
0 & M  &  0  \\
\cdots & \cdots & \ddots &\ddots  \\
0 &  \cdots & 0 & M & 0
\end{pmatrix} \equiv  A^{(n)} +  M^{(n)},
\end{equation}
with $M=e_1 (a-b)^T$.

Let us give now some previous results

\begin{lem}\label{lema2}
Let $(A,b)$ be an $s$-stage Runge-Kutta method of local order $p\ge 2$ and $M^{(n)}$ given by 
\eqref{matrizm}.
\begin{itemize}
\item 
For any  $ns$ dimensional vector $v=(v_j) \in \R^{ns}$ with $v_{s+1}= v_{2s+1} =  \cdots = v_{(n-1)s+1}=0$,
$v^T M^{(n)}=0$.
\item
If  $a^T c= 1/2$, then $M^{(n)} c^{(n)} =0$.
\item
If $b_1=0$ and $(b^T A)_1 =0$, then ${b^{(n)}}^T A^{(n)} M^{(n)}=0$.
\item
If $(a-b)^T A c= 0$, then $M^{(n)} A^{(n)} c^{(n)}  =0$.
\end{itemize}
\end{lem}
\begin{pf}
Since $v_{j s+1}=0$ for all $j$ and only the rows $j s+1$ of $M^{(n)}$ are non zero, it is immediate that 
$v^T M^{(n)}=0$.

The second item is immediate because $b^Tc = a^T c=1/2$.

For the third item, since the first component of $b^TA$ is zero, and the first component
of $b^T e b^T$ because $b_1=0$, then the vector $v= {b^{(n)}}^T A^{(n)}$ satisfies the conditions
$v_{js+1}=0$. form the first item we have $ {b^{(n)}}^T A^{(n)} M^{(n)}=0$.
The fourth item  follows from the fact that  $ e_1 (a-b)^T Ac =0$ together with $e_1 (a-b)^T e b^T c=0$.
\end{pf}

\begin{thm}
\label{orderthm}

Let $(A,b)$ be an $s$-stage Runge-Kutta method.
\begin{enumerate}
\item 
The RKRLS method has order $1$ if and only if the RK method $(A,b)$ has order $1$.
\item 
The RKRLS method has order $2$ if and only if the RK method $(A,b)$has order $2$.
\item
The RKLS method has order $3$ if and only if the RK method $(A,b)$ has order $3$ and either $b_1=0$ or else $a^T c=1/2$.
\item
The RKRLS method has order $4$ if and only if  the RK method $(A,b)$ has order 4 and at least one of the following conditions 
is satisfied,
\begin{itemize}
\item
$(b^T A)_1 =0$  and  $b_1=0$.
\item
$a^T c=1/2$ (the method $(A, a)$ has order 2) and $b_1=0$.
\item
$a^T c= 1/2$  and $a^T A c=1/6$  and $a^T c^2 =1/3$ (the method $(A, a)$ has order 3).
\end{itemize}
\item
Let the RK method have order $p\ge 5$. If $b_1=0$ and $(b^T A)_1 =0$ and  $a^T c=1/2$ the RKRLS method has order $\ge 5$.
\end{enumerate}
\end{thm}
\begin{pf} Let us proof each item.
\begin{enumerate}
\item
The RKRLS method has order 1 if and only if ${b^{(n)}}^T e^{(n)}=1$.  Since 
${b^{(n)}}^T e^{(n)}= b^T e$ the result follows.
\item
The modified method has order 2 if and only if ${b^{(n)}}^T e^{(n)} =1$ and ${b^{(n)}}^T \widehat c^{(n)} =1/2$. Since 
$\widehat c^{(n)} =c^{(n)}$, 
\[
{b^{(n)}}^T \widehat c^{(n)} ={b^{(n)}}^T c^{(n)} = \dfrac{1}{n} b^T c + \dfrac{n(n-1)}{2 n^2} b^Te
\]
which is equal to $1/2$ if and only if $b^T c= 1/2$.
\item
Let us suppose that the RK method has order $3$. The RKRLS method has order 2 by item 1. It has order 3 if ${b^{(n)}}^T {({\widehat c}^{(n)})}^2 =1/3$ and
${b^{(n)}}^T \widehat A^{(n)} \widehat c^{(n)} =1/6$.  The first condition is satisfied because the base method has order 3.
For the second one, 
${b^{(n)}}^T \widehat A^{(n)} \widehat c^{(n)} = {b^{(n)}}^T A^{(n)} c^{(n)}+ {b^{(n)}}^T M^{(n)} c^{(n)}$. Let us see that 
${b^{(n)}}^T M^{(n)} c^{(n)}=0$ and the order three of the base method implies that the RKRLS has also order 3.
If $b_1=0$, from the first item of Lemma \ref{lema2} , ${b^{(n)}}^T M^{(n)} =0$ which implies that ${b^{(n)}}^T M^{(n)} c^{(n)}=0$.
If $a^Tc=1/2$,  from the second item of Lemma ref{lema2}, $M^{(n)} c^{(n)}$ and therefore ${b^{(n)}}^T M^{(n)} c^{(n)}=0$.

Let us suppose now that the RKRLS method has order $3$. From the previous items the RK method has at least order $2$.
Since the RKRLS method has order $3$ it satisfies (taking $n=2$)
\[
\begin{array}{ll}
{b^{(2)}}^T (\widehat c^{(2)})^2 = &
{b^{(2)}}^T (c^{(2)})^2= \\
& \dfrac{2}{2^3} b^T c^2 + \dfrac{2}{2^3} b^T c + \dfrac{1}{2^3} b^T e = \\[7pt]
& \dfrac{1}{4} b^T c^2 + \dfrac{1}{4}  = 1/3
\end{array}
\]
which implies that $b^T c^2=1/3$. It also satisfies
\[
\begin{array}{ll}
{b^{(2)}}^T \widehat A^{(2)} \widehat c^{(2)} = &
{b^{(2)}}^T (A^{(2)}+ M^{(2)})  c^{(2)}= \\
& \dfrac{2}{2^3}b^T A c+ \dfrac{1}{2} b^T e b^T c + b^T e + \dfrac{1}{2^3}b^T e_1 (a-b)^T c = \\[10pt]
& \dfrac{1}{4} b^T A c + \dfrac{1}{8}+ \dfrac{1}{8} b_1 (a-b)^T c = 1/6
\end{array}
\]
and (taking $n=3$)
\[
\begin{array}{ll}
{b^{(3)}}^T \widehat A^{(3)} \widehat c^{(3)} = &
{b^{(3)}}^T (A^{(3)}+ M^{(3)})  c^{(3)}= \\[7pt]
& \dfrac{1}{9} b^T A c + \dfrac{4}{27}+ \dfrac{2}{27} b_1 (a-b)^T c = 1/6
\end{array}
\]
The above two equations have a unique solution
\[
b^T A c= 1/6, \qquad  b_1 (a-b)^T c =0
\]
which implies that the RK method has order $3$ and that either $b_1=0$ or else $a^T c= b^Tc =1/2$.
\item
Let us suppose that the RK method has order four. The RKRLS method has order 3 by item 2. It has order four if it satisfies the following order conditions
\[
\begin{array}{ll}
F(\tau)= f'''(f,f,f): & {b^{(n)}}^T {({\widehat c}^{(n)})}^3 =1/4, \\
F(\tau)= f'f''(f,f): & {b^{(n)}}^T \widehat A^{(n)} {({\widehat c}^{(n)})}^2 =1/12, \\
F(\tau)= f''(f'f,f): & {b^{(n)}}^T (\widehat A^{(n)}  \widehat c^{(n)} \cdot  \widehat c^{(n)}) =1/8, \\
F(\tau)= f'f'f'f: & {b^{(n)}}^T \widehat A^{(n)} \widehat A^{(n)} \widehat c^{(n)} =1/24.
\end{array}
\]

Using the conditions of order four for the base method $(A,b)$, we will have order four for the RKRLS method if
\begin{equation}
\label{order4}
\begin{array}{ll}
F(\tau)= f'f''(f,f): &  {b^{(n)}}^T M^{(n)} {({ c}^{(n)})}^2 =0, \\
F(\tau)= f''(f'f,f): &  (b^{(n)}\cdot  c^{(n)})^T (M^{(n)}   c^{(n)} ) =0, \\
F(\tau)= f'f'f'f: &  {b^{(n)}}^T  A^{(n)} M^{(n)}  c^{(n)} =0, \\
F(\tau)= f'f'f'f: &  {b^{(n)}}^T  M^{(n)} A^{(n)}  c^{(n)} =0.
\end{array}
\end{equation}

If $b_1=0$, the first, second  and fourth conditions of \eqref{order4} are satisfied by the first item of Lemma \ref{lema2}.
If $a^T c=1/2$, the third conditions is satisfied by the second item of Lemma \ref{lema2}. If 
$(b^T A)_s =0$ the third condition is satisfied by the third item of Lemma \ref{lema2}.

If $a^Tc=1/2$, the second and third order condition are satisfied. If $a^TAc=1/6$, the four condition is satisfied. Finally, if $a^Tc^2=1/3$, the fist condition is also satisfied.

Let us suppose now that the RKRLS method has order 4.  From previous items, the RK method has order 3 and either
$b_1=0$ or else $a^Tc=1/2$.  By the order four, ${b^{(2)}}^T (c^{(2)})^3=1/4$ and therefore $b^Tc^3=1/4$.
Also
\[
\begin{array}{l}
{b^{(2)}}^T \widehat A^{(2)} \widehat (c^{(2)})^2 = 
\dfrac{1}{9} b^T A c^2 + \dfrac{4}{27}+ \dfrac{2}{27} b_1 (a-b)^T c^2 = \dfrac{1}{12} \\[8pt]
{b^{(3)}}^T \widehat A^{(3)} \widehat (c^{(3)})^2 = 
\dfrac{1}{9} b^T A c^2 + \dfrac{4}{27}+ \dfrac{2}{27} b_1 (a-b)^T c^2 = \dfrac{1}{12}
\end{array}
\]
from where $b^TAc^2=1/12$ and  $b_1 (a-b)^T c^2=0$.
Taking now the order condition
\[
{b^{(2)}\cdot \widehat c^{(2)}}^T \widehat A^{(2)} \widehat c^{(2)}= 
\dfrac{1}{9} (b\cdot c)^T A c + \dfrac{4}{27} = \dfrac{1}{8}
\]
we get $(b\cdot c)^T A c =1/8$.

Finally, from
\[
\begin{array}{l}
{b^{(2)}}^T (\widehat A^{(2)})^2\widehat c^{(2)} = 
\dfrac{1}{8} b^T A^2 c + \dfrac{7}{192}+ \dfrac{1}{16} b_1 (a-b)^T A c \\[10pt]
\hskip 3cm + \dfrac{1}{16} b^T A e_1 (a-b)^T c= \dfrac{1}{24} \\[8pt]
{b^{(3)}}^T (\widehat A^{(3)})^2\widehat c^{(3)} = 
\dfrac{1}{27} b^T A^2 c + \dfrac{13}{351}+ \dfrac{2}{81} b_1 (a-b)^T A c \\[10pt]
\hskip 3cm + \dfrac{2}{81} b^T A e_1 (a-b)^T c= \dfrac{1}{24} 
\end{array}
\]
we obtain $b^TAAc=1/24$ and $b_1 (a-b)^T A c + b^T A e_1 (a-b)^T c =0$. But we know that $b_1=0$ or $(a-b)^Tc=0$.
If $b_1=0$ then  $b^T A e_1 (a-b)^T c =0$ and either $(b^TA)_1=0$ or else $a^T c=b^Tc=1/2$.
If $a^Tc=1/2$, then $b_1 (a-b)^TAc=0$ and either $b_1=0$ or $a^TAc= b^TAc=1/6$.

\item
The modified method has order 4 by item 3. It has order 5 if the order 5 conditions are satisfied. These can be
reduced to
\[
\begin{array}{ll}
F(\tau)= f'f'''(f,f,f): & {b^{(n)}}^T M^{(n)} {({ c}^{(n)})}^3 =0, \\
F(\tau)= f''(f''(f,f),f): & (b^{(n)}\cdot  c^{(n)})^T (M^{(n)}   {c^{(n)}}^2) =0, \\
F(\tau)= f'''(f'f,f,f): & (b^{(n)}\cdot  {c^{(n)}}^2)^T (M^{(n)}  c^{(n)}) =0, \\
F(\tau)= f'f'f''(f,f): & {b^{(n)}}^T  M^{(n)} A^{(n)}  {c^{(n)}}^2 =0 \\
F(\tau)= f''(f'f'f,f): & (b^{(n)}\cdot  c^{(n)})^T  (M^{(n)} A^{(n)}  c^{(n)}) =0 \\
F(\tau)= f'f'f'f'f: & {b^{(n)}}^T  M^{(n)} A^{(n)}  A^{(n)} c^{(n)}) =0, \\
F(\tau)= f'f'f'f'f: & {b^{(n)}}^T  M^{(n)} M^{(n)}  A^{(n)} c^{(n)}) =0, \\
F(\tau)= f'f'f'f'f: & {b^{(n)}}^T  M^{(n)} A^{(n)}  M^{(n)} c^{(n)}) =0, \\
F(\tau)= f'f'f'f'f: & {b^{(n)}}^T  M^{(n)} M^{(n)}  M^{(n)} c^{(n)}) =0, \\
F(\tau)= f''(f'f,f'f): & {b^{(n)}}^T  (A^{(n)} c^{(n)}) \cdot ( M^{(n)} c^{(n)})=0 \\
F(\tau)= f''(f'f,f'f): & {b^{(n)}}^T  (M^{(n)} c^{(n)}) \cdot ( M^{(n)} c^{(n)})=0 \\
F(\tau)= f'f'f'f'f: & {b^{(n)}}^T  A^{(n)} A^{(n)}  M^{(n)} c^{(n)}=0, \\
F(\tau)= f'f'f'f'f: & {b^{(n)}}^T  A^{(n)} M^{(n)}  M^{(n)} c^{(n)}=0, \\
F(\tau)= f''(f'f'f,f): & (b^{(n)}\cdot  c^{(n)})^T  (A^{(n)} M^{(n)}  c^{(n)})=0 \\
F(\tau)= f'f'f''(f,f): & {b^{(n)}}^T  A^{(n)} M^{(n)}  {c^{(n)}}^2 =0 \\
F(\tau)= f'f'f'f'f: & {b^{(n)}}^T  A^{(n)} M^{(n)}  A^{(n)} c^{(n)} =0.
\end{array}
\]
The first item of Lemma \ref{lema2} implies that the first nine conditions are satisfied. The second item
of Lemma \ref{lema2} implies that conditions ten to fourteen are satisfied.  Finally, the third condition of 
Lemma \ref{lema2} implies that the remaining conditions are also satisfied.

\end{enumerate}
\end{pf}

\section{RK methods reusing the last stage with orders 2, 3, 4 and 5}
In this section we show that, there exist standard RK methods  with $c_s=1$,  order $p= 2,3,4,5$  and minimum number of stages $s=2,3,4,6$ whose associated RKRLS methods ($s-1$ effective stages) have the same order $p$. 

\subsection{RKRLS methods with $s=2$ and order 2}
There exists a one-parameter family of RK methods with $s=2$ stages and order $p=2$ given by
\[
\begin{array}{c|cc}
0 & \\
c_2 & c_2  \\
\hline
&  b_1  & b_2
\end{array}
\qquad  
b_1 = 1 - b_2,  \qquad  b_2= 1/ c_2.
\]
With the condition $c_2=1$ ($b_2= 1/2$, $b_1=1/2$) there is a unique method.  Acccording the results in Theorem
\ref{orderthm}, the associated RKRLS also has order 2.

\subsection{RKRLS methods with $s=3$ and order 3}
There exists a one-parameter family of RK methods with $s=3$ stages, $c_3=1$ and order $p=3$ given by
\[
\begin{array}{c|ccc}
0 & \\
c_2 & c_2  \\
1 & 1-a_{32} & a_{32}  \\
\hline
&  b_1  & b_2 & b_3
\end{array}
\qquad  
\begin{array}{ll}
b_2= \dfrac{1}{6c_2(1-c_2) }, \ &   a_{32}= \dfrac{(1-c_2)}{c_2(2-3c_2)}, \\[12pt]
b_3= \dfrac{2-3 c_2}{6(1-c_2)},  &  b_1 = \dfrac{3c_2-1}{6c_2}.
\end{array}
\]
None of these methods satisfies the condition $(b^TA)_1=0$ and a unique method ($c_2=1/3$, $a_{32}=2$, $b_2=3/4$, $b_3=1/4$) satisfies the condition $b_1=0$. Acccording the results in Theorem \ref{orderthm}, the associated RKRLS also has order 3.
\subsection{RKRLS methods with $s=4$ and order 4}
It is known that any RK method with four stages and order four must satisfy the simplifying condition  $b^T a = b^T - (b\cdot c)^T$ (usually denoted by $D(1)$) and therefore the last node must be $c_4=1$.  This implies that $(b^TA)_1=0$ for any method of order four and four stages. There are four families of methods of order four (see \cite{Butcher:16}).

One of the families is obtained with $c_2=c_3$ and is given by the Butcher tableau
\[
\begin{array}{c|cccc}
0 & \\
\dfrac{1}{2} & \dfrac{1}{2}  \\[5pt]
\dfrac{1}{2} & \dfrac{a_{43}-1}{2a_{43}} & \dfrac{1}{2a_{43}} \\[8pt]
1  & 0  &  1-a_{43}  &  a_{43}  \\[4pt]
\hline
\vrule height 16pt width 0pt&  \dfrac{1}{6}  & \dfrac{2-a_{43}}{3} & \dfrac{a_{43}}{3}  & \dfrac{1}{6}
\end{array}
\]
In this family, $b_1 \ne 0$, but $(Ac)_4=1/2$. Therefore, according to Theorem \ref{orderthm} the associated RKRLS methods will have order 3. 
The classical RK, obtained for $a_{43}=1$,  belongs to this family and its associated RKRLS has order 3.

Another family is obtained with the condition $c_3 (1-c_2)(c_2-c_3) \ne 0$. It is a two-parameter family depending on $c_2$ and $c_3$ and with coefficients
\[
\begin{array}{ll}
b_2= \dfrac{2c_3-1}{12 c_2 (1-c_2)(c_3-c_2)}, & a_{32}= \dfrac{c_3(c_3-c_2)}{2c_2(1-2c_2)},  \\[12pt]
b_3= \dfrac{1-2c_2}{12 c_3 (1-c_2)(c_3-c_2)}, &  a_{42}= \dfrac{(1-c_2) (c_2+5c_3-4c_3^2-2)}{
  2 c_2 (c_3-c_2)(3+6c_2 c_3 -4c_2 -4c_3)} ,
 \\[10pt]
 b_4= \dfrac{3+6c_2 c_3 -4c_2 -4c_3)}{12(1-c_2)(1-c_3)}, & a_{43} = \dfrac{(1-c_2) (1-c_3)(1-2c_2)}{
  c_3(c_3-c_2)(3+6c_2 c_3 -4c_2 -4c_3)} .
\end{array}
\]
The $3/8$ RK method belongs to this family. For this method $b_1=1/8 \ne 0$ and $(Ac)_4 = 1/3 \ne 1/2$. Theorem
\ref{orderthm} guarantees order 2 for the associated RKRLS.

There is a subfamily, depending on one parameter $c_2$, with $b_1=0$ given by
\[
\begin{array}{ll}
c_3=\dfrac{2 c_2-1}{6 c_2-2},  &  a_{32}= \dfrac{1 - 4 c_2 + 6 c_2^2}{8 c_2(1 - 3 c_2)^2}, \\[12pt]
b_2= \dfrac{1}{6 - 30 c_2 + 60 c_2^2 - 36 c_2^3}, &   a_{42}= \dfrac{-(-1 + c_2)^2 (1 - 8 c_2 + 18 c_2^2)}{
  2 c_2 (1 - 10 c_2 + 36 c_2^2 - 60 c_2^3 + 36 c_2^4)}, \\[12pt]
b_3= \dfrac{2 (-1 + 3 c_2)^3}{3 (-1 + 4 c_2) (1 - 4 c_2 + 6 c_2^2)}, &  
a_{43} = \dfrac{2 (1 - 3 c_2)^2 (-1 + c_2) (-1 + 4 c_2)}{1 - 10 c_2 + 36 c_2^2 - 
 60 c2^3 + 36 c2^4},  \\[10pt]
 b_4= \dfrac{1 - 6 c_2 + 6 c_2^2}{6 - 30 c_2 + 24 c_2^2}.
\end{array}
\]
Acccording the results in Theorem \ref{orderthm}, the associated RKRLS methods also have order 4.

\subsection{RKRLS methods with $s=6$ and order 5}
Let us consider the three-parameter family of methods, depending on the coefficients $c_3, c_4$ and $c_5$, with $s=6$ stages and order $p=5$ obtained in \cite{DoPri:80}.
This family was obtained by imposing $c_6=1$ and the simplifying conditions $D(1)$
\[
b^T A = b^T - (b\cdot c)^T,
\]
and the modified $C(2)$
\[
Ac = c^2/2- (c_2^2/2) e_1.
\]
Consequently, all the methods satisfy $(b^T A)_1=0$ and $(Ac)_6=1/2$. The associated RKRLS methods will have order 3,
In particular for the DOPRI method of order 5, for which $b_1= 35/384 \ne 0$, the associated RKRLS has order 3.
However the subfamily of methods with
\[
c_5 = \dfrac{3 - 5 c_3 - 5 c_4 + 10 c_3 c_4}{5 - 10 c_3 - 10 c_4 + 30 c_3 c_4},
\]
satisfies $b_1=0$ which gives a two-parameter family of RKRLS methods with order 5.

\section{Numerical experiments}
In this section we present some numerical experiments to verify the results obtained in sections 2 and 3.  In the experiments, we will use the following methods
\begin{description}
\item[New2]
The only method with $s=2$, $p=2$ and $c_2=1$. The associated RKRLS  has order 2
\item[New3]
The only method with $s=3$, $p=3$, $c_3=1$ and $b_1=0$. The associated RKRLS has order 3
\item[RK-3/8]
The $3/8$ RK method with $s=4$, $p=4$, $c_2= 1/3$ and $c_3= 2/3$. The associated RKRLS has order 2
\item[RKClassic]
The classical RK with $s=4$, $p=4$, $c_2=c_3=1/2$. The associated RKRLS has order 3
\item[New4]
The method with with $s=4$, $p=4$, $b_1=0$ and $c_2= 1/6$. The associated RKRLS has order 4
\item[DOPRI54]
The Dormand and Prince method with with $s=6$, $p=5$,  The RKRLS has order 3
\item[New5]
The method with with $s=6$, $p=4$, $b_1=0$, and 
\[
\begin{array}{l}
c_3= 0.1574989977372333627197954851966028754675,  \\
c_4= 0.5649477718721229393448029991747476923492, \\ 
c_5= 0.62386437635858023903237445638618758645.
\end{array}
\]
The associated RKRLS has order 5.
\end{description}

With all these methods we have solved the Initial Value Problems 
\begin{enumerate}
\item
The two-body problem with excentricity  $e=0.1$ (problem D1 from the DETEST package \cite{detest,detest1}) integrated along the interval $[0,20]$
\[
\begin{array}{lll}
 x'' = -\dfrac{x}{\sqrt{(x^2+y^2)^3}},&  x(0)= 1-e, &  x'(0)=0, \\[10pt]
 y'' = -\dfrac{y}{\sqrt{(x^2+y^2)^3}}, & y(0)= 0, & y'(0)=\dfrac{\sqrt{1+e}}{\sqrt{1-e}}.
 \end{array}
\]
\item
The problem A3 from the DETEST package \cite{detest,detest1} integrated along the interval $[0, 20]$
\[
y' =y\cos t, \qquad  y(0)=1.
\]
\end{enumerate}

Using step sizes  $h= 0.2/2^j, j= 0, \ldots, 5$ we have computed, for each method and problem, the global errors GE obtained at the end point of the integration interval.
In Figures \ref{figure1} and \ref{figure2} we display  the global error GE against the step size used in double logarithmic scale. It is clearly observed that in all cases, the slope of the curves coincides with the order predicted by Theorem \ref{orderthm}.

\begin{figure}
\centering
\includegraphics[width=0.488\textwidth]{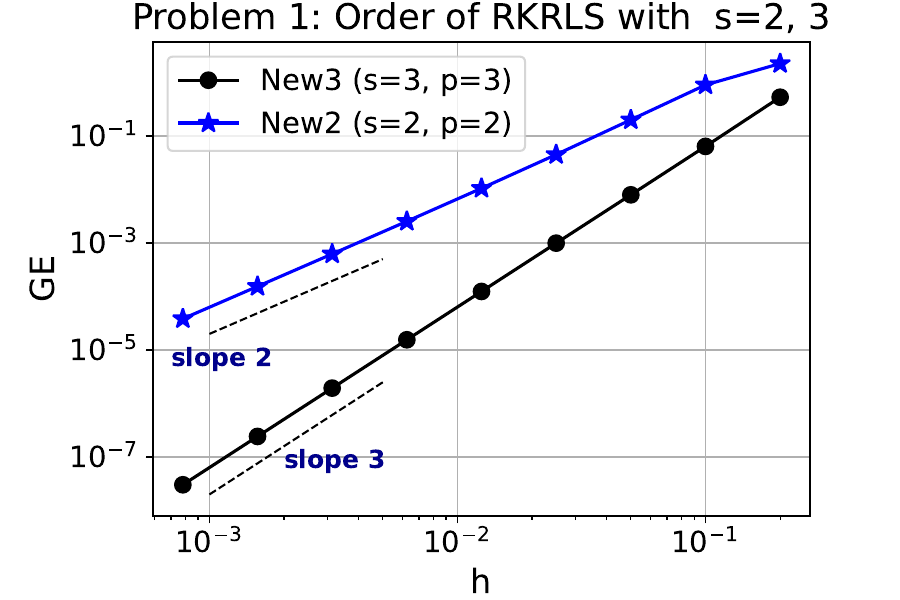}
\includegraphics[width=0.488\textwidth]{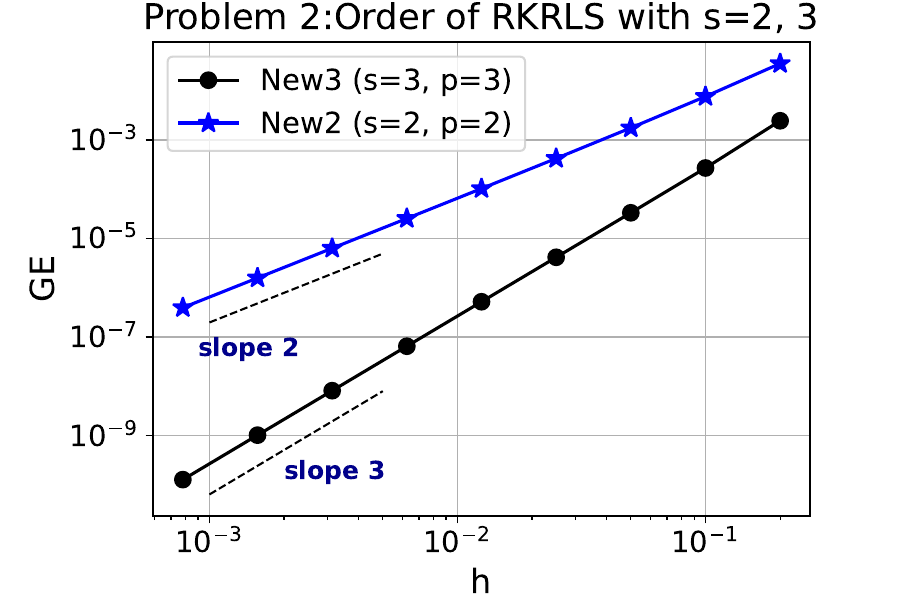}
\caption{Convergence orders of RKRLS methods with 2 and 3 stages, for problems 1 and 2}
\label{figure1}
\end{figure}

\begin{figure}
\centering
\includegraphics[width=0.488\textwidth]{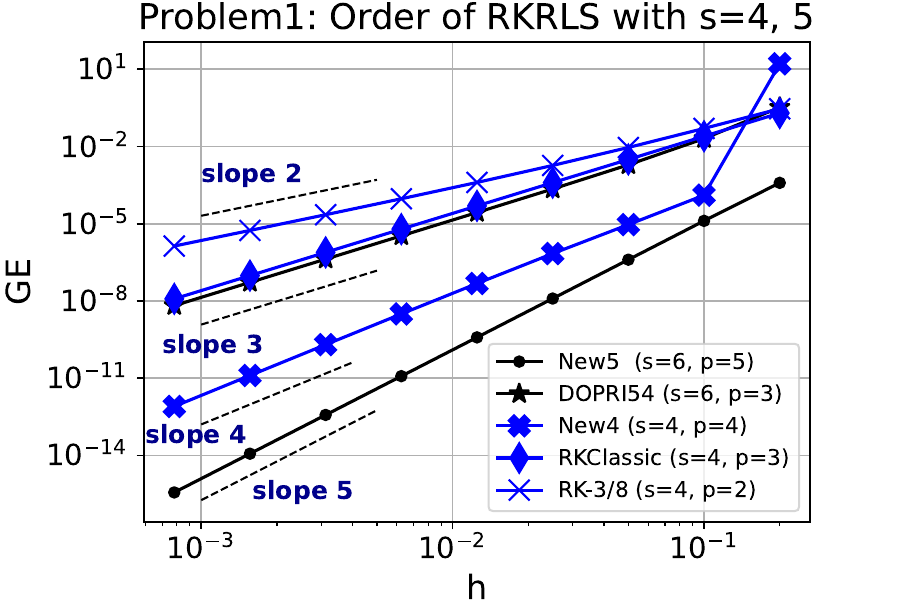}
\includegraphics[width=0.488\textwidth]{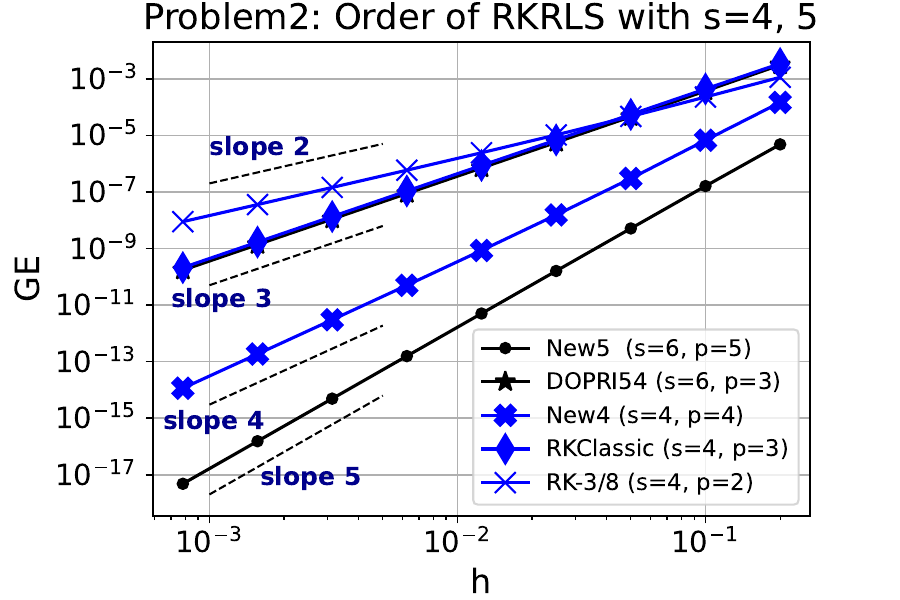}
\caption{Convergence orders of RKRLS methods with 4 and 5 stages, for problems 1 and 2}
\label{figure2}
\end{figure}

When you reduce the computational cost of a given method by modifying it, it is not unusual to pay some price, either in stability or else in accuracy. In order to  test this for RKRLS methods we have computed also the  golbal errors given by the underlying methods without reusing the last stage. In Tables \ref{table1} and \ref{table2} we collect these values together with the ones given by the RKRLS method for problem 1 and in Tables \ref{table3} and \ref{table4} the results for problem 2.  It is clear from the results that the accuracy of RKRLS methods give very similar global errors than their counterpart RK methods.  It is remarkable the case of problem 1 integrated with the method RKRLS of order 4 with time step $h=0.2$. The global error is very large, which can be due to stability. The RKRLS method has probably a smaller stability region than the associated RK method. The effects of reusing the last stage on stability of the method is currently  being studied object.

\begin{table}
\caption{Comparison of global errors of RKRLS and RK methods of orders $2$ and $3$ for problem 1}
\label{table1}
\begin{center}
\begin{tabular}{ |l|l|l|l|l|} 
\multicolumn{5}{c}{Problem 1}  \\
\hline
&\multicolumn{2}{|c|}{Order $p=2$} & \multicolumn{2}{|c|}{Order $p=3$} \\
\hline
$h$      &  RK     &  RKRLS  & RKRLS  & RK \\
\hline
0.2      & 1.897 e-00   & 2.255 e-00   & 5.137 e-01  &  5.297 e-01 \\ 
0.1      & 5.249 e-01   & 8.977 e-01   & 6.350 e-02  &  6.382 e-02 \\ 
0.05     & 1.134 e-01   & 2.002 e-01   & 7.931 e-03  &  7.942 e-03 \\ 
0.025    & 2.557 e-02   & 4.490 e-02   & 9.925 e-04  &  9.929 e-04 \\
0.0125   & 6.030 e-03   & 1.052 e-02   & 1.242 e-04  &  1.242 e-04 \\
0.00625  & 1.461 e-03   & 2.539 e-03   & 1.553 e-05  &  1.553 e-05 \\
 \hline
\end{tabular}
\end{center}
\end{table}

\begin{table}
\caption{Comparison of global errors of RKRLS and RK methods of orders $4$ and $5$}
\label{table2}
\begin{center}
\begin{tabular}{ |l|l|l|l|l|} 
\multicolumn{5}{c}{Problem 1}  \\
\hline
&\multicolumn{2}{|c|}{Order $p=4$} & \multicolumn{2}{|c|}{Order $p=5$} \\
\hline
$h$      &  RK     &  RKRLS  & RKRLS  & RK \\
\hline
 0.2     & 8.529 e-03 & 1.632 e+02 & 4.266 e-04  &  3.872 e-04 \\ 
 0.1     & 3.735 e-04 & 1.284 e-04 & 1.397 e-05  &  1.292 e-05 \\ 
 0.05    & 1.770 e-05 & 9.233 e-06 & 4.449 e-07  &  4.011 e-07 \\ 
 0.025   & 9.343 e-07 & 6.910 e-07 & 1.401 e-08  &  1.238 e-08 \\
0.0125   & 5.552 e-08 & 4.757 e-08 & 4.393 e-10  &  3.841 e-10 \\
0.00625  & 3.379 e-09 & 3.123 e-09 & 1.375 e-11  &  1.195 e-11 \\
 \hline
\end{tabular}
\end{center}
\end{table}

\begin{table}
\caption{Comparison of global errors of RKRLS and RK methods of orders $2$ and $3$ for problem 1}
\label{table3}
\begin{center}
\begin{tabular}{ |l|l|l|l|l|} 
\multicolumn{5}{c}{Problem 2}  \\
\hline
&\multicolumn{2}{|c|}{Order $p=2$} & \multicolumn{2}{|c|}{Order $p=3$} \\
\hline
$h$      &  RK     &  RKRLS  & RKRLS  & RK \\
\hline
0.2      & 2.600 e-02   & 3.545 e-02   & 2.122 e-03  &  2.439 e-03 \\ 
0.1      & 5.880 e-03   & 7.599 e-03   & 6.133 e-04  &  2.685 e-04 \\ 
0.05     & 1.396 e-03   & 1.751 e-03   & 3.279 e-05  &  3.306 e-05 \\ 
0.025    & 3.397 e-04   & 4.196 e-04   & 4.118 e-06  &  4.130 e-06 \\
0.0125   & 8.375 e-05   & 1.026 e-04   & 5.163 e-07  &  5.169 e-07 \\
0.00625  & 2.079 e-05   & 2.536 e-05   & 6.465 e-08  &  6.468 e-08 \\
 \hline
\end{tabular}
\end{center}
\end{table}

\begin{table}
\caption{Comparison of global errors of RKRLS and RK methods of orders $4$ and $5$}
\label{table4}
\begin{center}
\begin{tabular}{ |l|l|l|l|l|} 
\multicolumn{5}{c}{Problem 2}  \\
\hline
&\multicolumn{2}{|c|}{Order $p=4$} & \multicolumn{2}{|c|}{Order $p=5$} \\
\hline
$h$      &  RK     &  RKRLS  & RKRLS  & RK \\
\hline
 0.2     & 5.057 e-05 & 1.464 e-04 & 5.690 e-06  &  4.787 e-06 \\ 
 0.1     & 2.452 e-06 & 6.708 e-06 & 1.832 e-07  &  1.607 e-07 \\ 
 0.05    & 1.595 e-07 & 2.986 e-07 & 5.752 e-09  &  5.104 e-09 \\ 
 0.025   & 1.018 e-08 & 1.530 e-08 & 1.797 e-10  &  1.600 e-10 \\
0.0125   & 6.431 e-10 & 8.495 e-10 & 5.610 e-12  &  5.002 e-12 \\
0.00625  & 4.040 e-11 & 4.974 e-11 & 1.752 e-13  &  1.563 e-13 \\
 \hline
\end{tabular}
\end{center}
\end{table}

\section{Conclusions}
We have studied the order of RK method
in which the last function evaluation  in a step
is reused (RKRLS), substituting the first evaluation for the next step, 
by means of the composition of RK methods and the adjoint (also named reverse) 
of an RK method. 

We have proved that if the composed RKRLS methods have order $p$,
the the RKRLS method has also order $p$.
We have obtained the conditions that the coefficients of an explicit RK method
must satisfy to get an RKRLS method with the same order.
With a detailed study of the families of $s$--stage explicit RK methods with the highest order $p=s$ 
we have proved that with some simple additional conditions on the coefficients
the associated RKRLS methods retain  the order of the original RK method. This is relevant because
we have proved that it is possible to maintain the order reducing the 
computational cost from $s$ to $(s-1)$ evaluations of the vector field per step.

The results of some numerical experiments have confirmed the theoretical results.

%
%  REFERENCES
%

\baselineskip=0.9\normalbaselineskip

\bibliographystyle{model1b-num-names}
{

}

\end{document}